\documentclass[12pt]{article}
\usepackage{amssymb}
\usepackage{amsthm}
\usepackage{amsmath}
\usepackage{graphicx}
\usepackage{enumerate}

\DeclareMathOperator{\Max}{Max}
\DeclareMathOperator{\Min}{Min}

\newtheorem{theorem}{Theorem}
\newtheorem{definition}[theorem]{Definition}
\newtheorem{lemma}[theorem]{Lemma}

\newtheorem{remark}[theorem]{Remark}
\newtheorem{example}[theorem]{Example}
\newtheorem{corollary}[theorem]{Corollary}
\title{Operator residuation in orthomodular posets of finite height}
\author{Ivan~Chajda and Helmut~L\"anger}
\date{}
\begin{document}

\footnotetext{Support of the research of both authors by the Austrian Science Fund (FWF), project I~4579-N, and the Czech Science Foundation (GA\v CR), project 20-09869L, entitled ``The many facets of orthomodularity'', is gratefully acknowledged.}

\maketitle

\begin{abstract}
We show that for every orthomodular poset $\mathbf P=(P,\leq,{}',0,1)$ of finite height there can be defined two operators forming an adjoint pair with respect to an order-like relation defined on the power set of $P$. This enables us to introduce the so-called operator residuated poset corresponding to $\mathbf P$ from which the original orthomodular poset $\mathbf P$ can be recovered. Moreover, this correspondence is almost one-to-one. We show that this construction of operators can be applied also to so-called weakly orthomodular and dually weakly orthomodular posets. Examples of such posets are included.
\end{abstract}



The first attempts for algebraic axiomatization of the logic of quantum mechanics were done by G.~Birkhoff and J.~von Neumann \cite{BV} and K.~Husimi \cite H via orthomodular lattices. Later on the people working in quantum mechanics decided that such an axiomatization is imprecise since for two operators $P$ and $Q$ on the corresponding Hilbert space their disjunction, i.e.\ the lattice join, need not exist in the case when $P$ and $Q$ are neither comparable nor orthogonal. Hence the concept of an orthomodular poset was introduced (see \cite F) and the majority of researchers believe that orthomodular posets form a better algebraic axiomatization of the logic of quantum mechanics, see, e.g., \cite{MP} and \cite{PF}.

It is not so easy to introduce the connective implication in the propositional logic based on an orthomodular poset. Namely, since orthomodular posets are only partial lattices where the lattice operation join need not exist for elements being neither comparable nor orthogonal, also implication is often considered as a partial operation. However, such an approach may cause problems for developing the corresponding logic. We used another approach where implication is everywhere defined, but its result $x\rightarrow y$ for given entries $x$ and $y$ need not be an element, but may be a subset of the orthomodular poset in question, see \cite{CL21}. In \cite{CL21} we showed that for such an implication there exists an operator $\odot$ (considered as conjunction) such that $(\rightarrow,\odot)$ forms a pair connected by some kind of unsharp adjointness. The disadvantage of this approach is that the rule for unsharp adjointness is rather complicated and does not correspond to adjointness as usually considered in residuated structures.

In \cite{CL22} we developed a better approach where the considered orthomodular posets are of finite height. We introduced an everywhere defined connective implication which satisfies the properties usually asked in classical and non-classical logics. Again the result $x\rightarrow y$ for given entries $x$ and $y$ need not be an element, but may be a subset of the orthomodular poset in question. But now this subset contains only the maximal values. Of course, if there is only one maximal value then the result is equal to it and hence the result of implication is an element. In other words, if our orthomodular poset is a lattice then the implication defined there coincides with the implication usually introduced in orthomodular lattices, see, e.g., \cite K and \cite{MP}. On the other hand, if several maximal values of $x\rightarrow y$ exist then this may not cause problems since such an ``unsharp'' reasoning in quantum mechanics was already treated e.g.\ in \cite{GG}. Moreover, this approach enabled us to derive a nice Gentzen system for the axiomatization of this logic, i.e.\ a system consisting of five simple axioms and six derivation rules only, see \cite{CL22} for details.

The above mentioned approach motivates us to study a kind of residuation for such a logic based on orthomodular posets of finite height. The aim of this paper is to find an operator $\odot$ such that the couple $(\rightarrow,\odot)$ forms an adjoint pair.

We start with definitions and basic properties of our concepts.

\begin{definition}\label{def3}
An {\em orthomodular poset} is a bounded poset $(P,\leq,{}',0,1)$ with an antitone involution $'$ which is a complementation satisfying the following conditions:
\begin{enumerate}[{\rm(i)}]
\item If $x,y\in P$ and $x\leq y'$ then $x\vee y$ is defined.
\item If $x,y\in P$ and $x\leq y$ then $y=x\vee(y'\vee x)'$ {\rm(}{\em orthomodularity}{\rm)}.
\end{enumerate}
\end{definition}

If $x,y\in P$ and $x\leq y'$ then $x$ and $y$ are called {\em orthogonal} to each other which will be denoted by $x\perp y$. It is easy to see that the expression at the end of condition (ii) is well-defined. Due to De Morgan's laws, (ii) can equivalently be rewritten in anyone of the following ways:
\begin{enumerate}
\item[(A)] If $x,y\in P$ and $x\leq y$ then $y=x\vee(y\wedge x')$.
\item[(B)] If $x,y\in P$ and $x\leq y$ then $x=y\wedge(x\vee y')$.
\end{enumerate}
If, moreover, the poset $(P,\leq)$ is a lattice, then $(P,\leq,{}',0,1)$ is called an {\em orthomodular lattice}.

Observe that if a poset $(P,\leq,{}')$ with an involution satisfies (B) then it also satisfies (i).

\begin{example}
Of course, every orthomodular lattice is an orthomodular poset. However, there are orthomodular posets that are not lattices. The smallest orthomodular poset that is not a lattice is depicted in Figure~1:

\begin{center}
\setlength{\unitlength}{7mm}
\begin{picture}(18,8)
\put(9,1){\circle*{.3}}
\put(1,3){\circle*{.3}}
\put(3,3){\circle*{.3}}
\put(5,3){\circle*{.3}}
\put(7,3){\circle*{.3}}
\put(9,3){\circle*{.3}}
\put(11,3){\circle*{.3}}
\put(13,3){\circle*{.3}}
\put(15,3){\circle*{.3}}
\put(17,3){\circle*{.3}}
\put(1,5){\circle*{.3}}
\put(3,5){\circle*{.3}}
\put(5,5){\circle*{.3}}
\put(7,5){\circle*{.3}}
\put(9,5){\circle*{.3}}
\put(11,5){\circle*{.3}}
\put(13,5){\circle*{.3}}
\put(15,5){\circle*{.3}}
\put(17,5){\circle*{.3}}
\put(9,7){\circle*{.3}}
\put(1,3){\line(0,2)2}
\put(1,3){\line(1,1)2}
\put(1,3){\line(3,1)6}
\put(1,3){\line(4,1)8}
\put(3,3){\line(-1,1)2}
\put(3,3){\line(1,1)2}
\put(3,3){\line(2,1)4}
\put(3,3){\line(4,1)8}
\put(5,3){\line(-1,1)2}
\put(5,3){\line(0,1)2}
\put(5,3){\line(2,1)4}
\put(5,3){\line(3,1)6}
\put(7,3){\line(-3,1)6}
\put(7,3){\line(-2,1)4}
\put(7,3){\line(3,1)6}
\put(7,3){\line(4,1)8}
\put(9,3){\line(-4,1)8}
\put(9,3){\line(-2,1)4}
\put(9,3){\line(2,1)4}
\put(9,3){\line(4,1)8}
\put(11,3){\line(-4,1)8}
\put(11,3){\line(-3,1)6}
\put(11,3){\line(2,1)4}
\put(11,3){\line(3,1)6}
\put(13,3){\line(-3,1)6}
\put(13,3){\line(-2,1)4}
\put(13,3){\line(0,1)2}
\put(13,3){\line(1,1)2}
\put(15,3){\line(-4,1)8}
\put(15,3){\line(-2,1)4}
\put(15,3){\line(-1,1)2}
\put(15,3){\line(1,1)2}
\put(17,3){\line(-4,1)8}
\put(17,3){\line(-3,1)6}
\put(17,3){\line(-1,1)2}
\put(17,3){\line(0,1)2}
\put(9,1){\line(-4,1)8}
\put(9,1){\line(-3,1)6}
\put(9,1){\line(-2,1)4}
\put(9,1){\line(-1,1)2}
\put(9,1){\line(0,1)2}
\put(9,1){\line(1,1)2}
\put(9,1){\line(2,1)4}
\put(9,1){\line(3,1)6}
\put(9,1){\line(4,1)8}
\put(9,7){\line(-4,-1)8}
\put(9,7){\line(-3,-1)6}
\put(9,7){\line(-2,-1)4}
\put(9,7){\line(-1,-1)2}
\put(9,7){\line(0,-1)2}
\put(9,7){\line(1,-1)2}
\put(9,7){\line(2,-1)4}
\put(9,7){\line(3,-1)6}
\put(9,7){\line(4,-1)8}
\put(8.85,.25){$0$}
\put(.85,2.3){$a$}
\put(2.85,2.3){$b$}
\put(4.85,2.3){$c$}
\put(6.85,2.3){$d$}
\put(8.85,2.3){$e$}
\put(10.85,2.3){$f$}
\put(12.85,2.3){$g$}
\put(14.85,2.3){$h$}
\put(16.85,2.3){$i$}
\put(.8,5.4){$i'$}
\put(2.8,5.4){$h'$}
\put(4.8,5.4){$g'$}
\put(6.8,5.4){$f'$}
\put(8.8,5.4){$e'$}
\put(10.8,5.4){$d'$}
\put(12.8,5.4){$c'$}
\put(14.8,5.4){$b'$}
\put(16.8,5.4){$a'$}
\put(8.85,7.4){$1$}
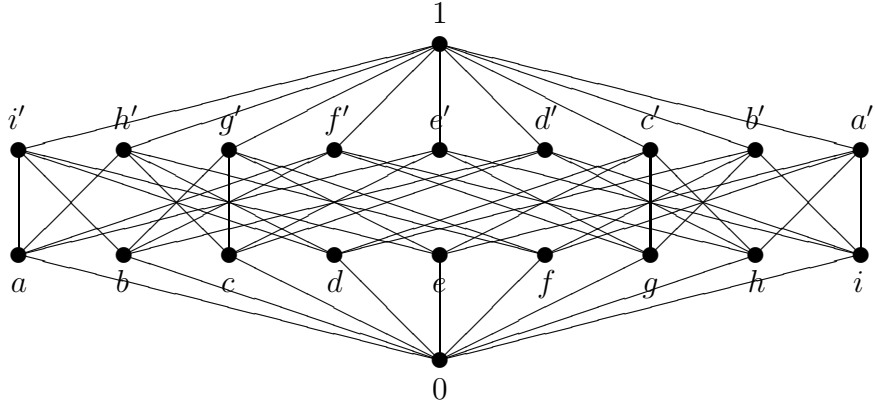
\put(1.8,-.75){{\rm Fig.~1}. Smallest orthomodular poset not being a lattice}
\end{picture}
\end{center}

\vspace*{3mm}

It is not a lattice since $a\vee b$ does not exist because $i'$ and $f'$ are two distinct minimal upper bounds of $a$ and $b$. This orthomodular poset is finite and hence of finite height. The smallest orthomodular lattice that is not a Boolean algebra is visualized in Figure~2:

\begin{center}
\setlength{\unitlength}{7mm}
\begin{picture}(8,6)
\put(4,1){\circle*{.3}}
\put(1,3){\circle*{.3}}
\put(3,3){\circle*{.3}}
\put(5,3){\circle*{.3}}
\put(7,3){\circle*{.3}}
\put(4,5){\circle*{.3}}
\put(4,1){\line(-3,2)3}
\put(4,1){\line(-1,2)1}
\put(4,1){\line(1,2)1}
\put(4,1){\line(3,2)3}
\put(4,5){\line(-3,-2)3}
\put(4,5){\line(-1,-2)1}
\put(4,5){\line(1,-2)1}
\put(4,5){\line(3,-2)3}
\put(3.85,.3){$0$}
\put(.35,2.85){$a$}
\put(2.35,2.85){$b$}
\put(5.4,2.85){$b'$}
\put(7.4,2.85){$a'$}
\put(3.85,5.4){$1$}
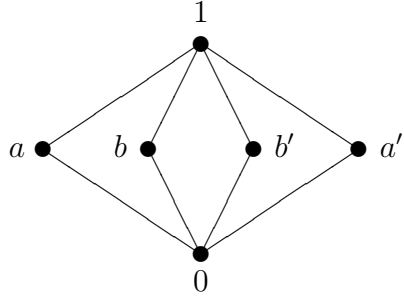
\put(-4.6,-.75){{\rm Fig.~2}. Smallest orthomodular lattice not being a Boolean algebra}
\end{picture}
\end{center}

\vspace*{3mm}

It is easy to show that a horizontal sum of an arbitrary number of orthomodular posets of finite height is an orthomodular poset of finite height again. Hence there exist also infinite orthomodular posets of finite height.
\end{example}

Let $(P,\leq)$ be a poset, $a,b\in P$ and $A,B\subseteq P$. We define $A\leq B$ if $x\leq y$ for all $x\in A$ and all $y\in B$. Instead of $\{a\}\leq\{b\}$, $\{a\}\leq B$ and $A\leq\{b\}$ we simply write $a\leq b$, $a\leq B$ and $A\leq b$, respectively. Further, we define
\begin{align*}
L(A) & :=\{x\in P\mid x\leq A\}, \\
U(A) & :=\{x\in P\mid A\leq x\}.
\end{align*}
Instead of $L(\{a\})$, $L(\{a,b\})$, $L(\{a\}\cup B)$, $L(A\cup B)$ and $L\big(U(A)\big)$ we simply write $L(a)$, $L(a,b)$, $L(a,B)$, $L(A,B)$ and $LU(A)$, respectively. Analogously, we proceed in similar cases. The meaning of expressions like $A\vee B$, $a\vee B$, $A'$ and similar ones is clear.

\begin{lemma}\label{lem2}
Let $\mathbf P=(P,\leq,{}')$ be a poset with a unary operation, $a,b\in P$ and $A,B\subseteq P$. Then the following hold:
\begin{enumerate}[{\rm(i)}]
\item If $\mathbf P$ satisfies {\rm(A)} and $a\leq B$ then $B=a\vee(B\wedge a')$.
\item If $\mathbf P$ satisfies {\rm(B)} and $A\leq b$ then $A=b\wedge(A\vee b')$.
\end{enumerate}
\end{lemma}

\begin{proof}
\
\begin{enumerate}[(i)]
\item If $\mathbf P$ satisfies (A) and $a\leq B$ then $B=\{x\mid x\in B\}=\{a\vee(x\wedge a')\mid x\in B\}=a\vee(B\wedge a')$.
\item If $\mathbf P$ satisfies (B) and $A\leq b$ then $A=\{x\mid x\in A\}=\{b\wedge(x\vee b')\mid x\in A\}=b\wedge(A\vee b')$.
\end{enumerate}
\end{proof}

A poset $(P,\leq)$ is said to be of {\em finite height} if it contains no infinite chain. In such a case, every non-empty subset of $P$ has at least one minimal and one maximal element. Denote by $\Min A$ and $\Max A$ the set of all minimal and maximal elements of the subset $A$ of $P$, respectively.

\begin{lemma}\label{lem1}
Let $\mathbf P=(P,\leq,{}')$ be a poset of finite height and $a,b\in P$. Then the following hold:
\begin{enumerate}[{\rm(i)}]
\item If $\mathbf P$ satisfies {\rm(A)} and $'$ is an involution then $\Min U(a,b')\wedge b$ is defined.
\item If $\mathbf P$ satisfies {\rm(B)} then $a'\vee\Max L(a,b)$ is defined.
\end{enumerate}
\end{lemma}

\begin{proof}
Assertion (i) follows from $b'\leq\Min U(a,b')$, and (ii) follows from $\Max L(a,b)\leq a$.
\end{proof}

Now on every orthomodular poset $(P,\leq,{}',0,1)$ of finite height we can introduce the following everywhere defined operators $\odot,\rightarrow\colon P^2\rightarrow2^P\setminus\{\emptyset\}$:
\begin{align*}
      x\odot y & :=\Min U(x,y')\wedge y, \\
x\rightarrow y & :=x'\vee\Max L(x,y)
\end{align*}
for all $x,y\in P$.

\begin{remark}\label{rem1}
If $\mathbf L=(L,\vee,\wedge,{}',0,1)$ is an orthomodular lattice then
\begin{align*}
\Min U(x,y') & =x\vee y', \\
 \Max L(x,y) & =x\wedge y,
\end{align*}
thus
\begin{align*}
      x\odot y & =(x\vee y')\wedge y, \\
x\rightarrow y & =x'\vee(x\wedge y).
\end{align*}
{\rm(}The expression $(x\vee y')\wedge y$ is called the {\em Sasaki projection} of $x$ to $[0,y]$.{\rm)} This means that $\rightarrow$ is the implication in orthomodular lattices as considered in {\rm\cite{Be}}, {\rm\cite K} and {\rm\cite{MP}}. In the case where $\mathbf L$ is a Boolean algebra {\rm(}axiomatizing classical propositional logic{\rm)} then the just defined operations reduce to
\begin{align*}
      x\odot y & =x\wedge y, \\
x\rightarrow y & =x'\vee y
\end{align*}
as expected. For orthomodular posets of finite height, the operator $\rightarrow$ introduced above coincides with the implication introduced in {\rm\cite{CL22}}.
\end{remark}

Let $(P,\leq)$ be a poset,  We define a binary relation $\sqsubseteq$ on $2^P\setminus\{\emptyset\}$ as follows: For all non-empty subsets $A$ and $B$ of $P$
\[
A\sqsubseteq B:\Leftrightarrow\text{ there exists some }a\in A\text{ and some }b\in B\text{ with }a\leq b
\]
Now let $a,b\in P$ and $A,B,C$ be non-empty subsets of $P$. Then the following hold:
\begin{enumerate}[(i)]
\item $\{a\}\sqsubseteq\{b\}$ if and only if $a\leq b$.
\item $A\sqsubseteq A$
\item $A\leq B$ implies $A\sqsubseteq B$.
\item $A\leq B\sqsubseteq C$ implies $A\sqsubseteq C$.
\item $A\sqsubseteq B\leq C$ implies $A\sqsubseteq C$.
\end{enumerate}
Statement (i) shows that if we identify the elements $x$ of $P$ with the corresponding singletons $\{x\}$ then the restriction of the relation $\sqsubseteq$ to $P$ coincides just with the original relation $\leq$ on $P$.

In the following, if $P$ is a non-empty set and $\otimes$ a mapping from $P^2$ to $2^P$, i.e.\ a binary operator on $P$, then we define
\[
A\otimes y:=\bigcup_{x\in A}(x\otimes y)
\]
for all $A\subseteq P$ and $y\in P$.

For the operators $\odot$ and $\rightarrow$ defined above we can prove the following properties.

\begin{lemma}
Let $(P,\leq,{}',0,1)$ be an orthomodular poset of finite height and $a,b,c\in P$ and define
\begin{align*}
      x\odot y & :=\Min U(x,y')\wedge y, \\
x\rightarrow y & :=x'\vee\Max L(x,y)
\end{align*}
for all $x,y\in P$. Then the following hold:
\begin{enumerate}[{\rm(i)}]
\item $\Min U(a,b)\odot a=a$
\item $a\leq b$ implies $a\rightarrow b=1$.
\item $1\rightarrow a=a$
\item $a\leq b$ implies $c\rightarrow a\sqsubseteq c\rightarrow b$.
\end{enumerate}
\end{lemma}

\begin{proof}
Because of Lemma~\ref{lem1}, $\odot$ and $\rightarrow$ are well-defined.
\begin{enumerate}[(i)]
\item $\Min U(a,b)\odot a=\bigcup\limits_{x\in\Min U(a,b)}(x\odot a)=\bigcup\limits_{x\in\Min U(a,b)}\big(\Min U(x,a')\wedge a\big)= \\
=\bigcup\limits_{x\in\Min U(a,b)}\{1\wedge a\}=\{a\}$
\item $a\leq b$ implies $a\rightarrow b=a'\vee\Max L(a,b)=a'\vee a=1$.
\item $1\rightarrow a=1'\vee\Max L(1,a)=0\vee a=a$
\item Everyone of the following statements implies the next one:
\begin{align*}
                               a & \leq b, \\
                          L(c,a) & \subseteq L(c,b), \\
                     \Max L(c,a) & \sqsubseteq\Max L(c,b), \\
c\rightarrow a=c'\vee\Max L(c,a) & \sqsubseteq c'\vee\Max L(c,b)=c\rightarrow b.
\end{align*}
\end{enumerate}
\end{proof}

Now we introduce our main concept.

\begin{definition}\label{def1}
An {\em operator residuated structure} is an ordered six-tuple $\mathbf R=(P,\leq,\odot,\rightarrow,0,1)$ satisfying the following conditions:
\begin{enumerate}[{\rm(i)}]
\item $(P,\leq,0,1)$ is a bounded poset.
\item $\odot$ and $\rightarrow$ are mappings from $P^2$ to $2^P\setminus\{\emptyset\}$.
\item $x\odot y\sqsubseteq z$ if and only if $x\sqsubseteq y\rightarrow z$.
\item $x\odot1\approx1\odot x\approx x$
\item $y\rightarrow0\leq x$ implies $x\odot y=x\wedge y$.
\item $y\leq x$ implies $x\rightarrow y=(x\rightarrow0)\vee y$.
\end{enumerate}
Condition {\rm(iii)} is called {\em operator left adjointness}. $\mathbf R$ is
\begin{itemize}
\item called {\em idempotent} if $x\odot x\approx x$,
\item called {\em divisible} if $(x\rightarrow y)\odot x\approx\Max L(x,y)$,
\item said to satisfy the {\em double negation law} if $(x\rightarrow0)\rightarrow0\approx x$.
\item said to satisfy the {\em contraposition law} if $x\leq y$ implies $y\rightarrow0\leq x\rightarrow0$.
\end{itemize}
\end{definition}

Now we can prove that every orthomodular poset of finite height can be converted into such an operator residuated structure.

\begin{theorem}\label{th1}
Let $\mathbf P=(P,\leq,{}',0,1)$ be an orthomodular poset of finite height and define mappings $\odot,\rightarrow\colon P^2\rightarrow2^P\setminus\{\emptyset\}$ by
\begin{align*}
      x\odot y & :=\Min U(x,y')\wedge y, \\
x\rightarrow y & :=x'\vee\Max L(x,y)
\end{align*}
for all $x,y\in P$. Then $\mathbb R(\mathbf P):=(P,\leq,\odot,\rightarrow,0,1)$ is an idempotent and divisible operator residuated structure satisfying both the double negation law as well as the contraposition law and $x\rightarrow0\approx x'$.
\end{theorem}

\begin{proof}
Let $a,b,c\in P$. Because of Lemma~\ref{lem1}, $\odot$ and $\rightarrow$ are well-defined. Moreover, we have $x\rightarrow0\approx x'\vee\Max L(x,0)=x'$.
\begin{enumerate}
\item[(i)] and (ii) are evident.
\item[(iii)] Anyone of the following statements implies the next one:
\begin{align*}
                                                                a\odot b & \sqsubseteq c, \\
                                                    \Min U(a,b')\wedge b & \sqsubseteq c, \\
         \text{There exists some }d\in\Min U(a,b')\wedge b\text{ with }d & \leq c, \\
         \text{There exists some }d\in\Min U(a,b')\wedge b\text{ with }d & \in L(b,c), \\
         \text{There exists some }d\in\Min U(a,b')\wedge b\text{ with }d & \sqsubseteq\Max L(b,c), \\
                                                    \Min U(a,b')\wedge b & \sqsubseteq \Max L(b,c), \\
                                    b'\vee\big(\Min U(a,b')\wedge b\big) & \sqsubseteq b'\vee\Max L(b,c), \\
                                                       a\leq\Min U(a,b') & \sqsubseteq b\rightarrow c, \\
                                                                       a & \sqsubseteq b\rightarrow c, \\
                                                                       a & \sqsubseteq b'\vee\Max L(b,c), \\
           \text{There exists some }e\in b'\vee\Max L(b,c)\text{ with }a & \leq e, \\   
           \text{There exists some }e\in b'\vee\Max L(b,c)\text{ with }e & \in U(a,b'), \\   
\text{There exists some }e\in b'\vee\Max L(b,c)\text{ with }\Min U(a,b') & \sqsubseteq e, \\   
                                                            \Min U(a,b') & \sqsubseteq b'\vee\Max L(b,c), \\
                                                    \Min U(a,b')\wedge b & \sqsubseteq\big(b'\vee\Max L(b,c)\big)\wedge b, \\
                                                                a\odot b & \sqsubseteq\Max L(b,c)\leq c, \\
                                                                a\odot b & \sqsubseteq c.
\end{align*}
In the eighth and sixteenth statement we use Lemma~\ref{lem2}.
\item[(iv)] We have $x\odot1\approx\Min U(x,1')\wedge1\approx x$ and $1\odot x\approx\Min U(1,x')\wedge x\approx x$.
\item[(v)] $b\rightarrow0\leq a$ implies $b'\leq a$ and hence $a\odot b=\Min U(a,b')\wedge b=a\wedge b$.
\item[(vi)] $b\leq a$ implies $a\rightarrow b=a'\vee\Max L(a,b)=a'\vee b=(a\rightarrow0)\vee b$.
\end{enumerate}
Finally, we have
\begin{align*}
                   x\odot x & \approx\Min U(x,x')\wedge x\approx x, \\
    (x\rightarrow y)\odot x & \approx\bigcup_{z\in x\rightarrow y}(z\odot x)\approx\bigcup_{z\in x'\vee\Max L(x,y)}\big(\Min U(z,x')\wedge x\big)\approx \\
                            & \approx\bigcup_{u\in\Max L(x,y)}\big(\Min U(x'\vee u,x')\wedge x\big)\approx\bigcup_{u\in\Max L(x,y)}\{(x'\vee u)\wedge x\}\approx \\
                            & \approx\bigcup_{u\in\Max L(x,y)}\{u\}\approx\Max L(x,y), \\
(x\rightarrow0)\rightarrow0 & \approx x''\approx x.
\end{align*}
In the last but one line we used orthomodularity. Finally, $a\leq b$ implies $b\rightarrow0=b'\leq a'=a\rightarrow0$.
\end{proof}

It is important that we can also prove the converse.

\begin{theorem}
Let $\mathbf R=(R,\leq,\odot,\rightarrow,0,1)$ be a divisible operator residuated structure of finite height satisfying both the double negation law as well as the contraposition law and define
\[
x':=x\rightarrow0
\]
for all $x\in R$. Then $\mathbb P(\mathbf R):=(R,\leq,{}',0,1)$ is an orthomodular poset.
\end{theorem}

\begin{proof}
Let $a,b\in R$. Since $\mathbf R$ satisfies both the double negation law as well as the contraposition law, $'$ is an antitone involution on $(R,\leq)$. Now assume $a\leq b$. Because of (vi) of Definition~\ref{def1} we have $b\rightarrow a=(b\rightarrow0)\vee a=b'\vee a\geq b'$. According to (v) of Definition~\ref{def1} we conclude $(b\rightarrow a)\odot b=(b\rightarrow a)\wedge b$ which together with divisibility yields
\[
b\wedge(a\vee b')=(b'\vee a)\wedge b=(b\rightarrow a)\wedge b=(b\rightarrow a)\odot b=\Max L(b,a)=a
\]
showing orthomodularity. Because of $0\leq a$ we obtain $0=a\wedge(0\vee a')=a\wedge a'$ completing the proof of the theorem.
\end{proof}

The following theorem shows that the correspondence between orthomodular posets of finite height and certain operator residuated structures of finite height is nearly one-to-one.

\begin{theorem}\label{th2}
\
\begin{enumerate}[{\rm(i)}]
\item Let $\mathbf P=(P,\leq,{}',0,1)$ be an orthomodular poset of finite height. Then $\mathbb P\big(\mathbb R(\mathbf P)\big)=\mathbf P$.
\item Let $\mathbf R=(R,\leq,\odot,\rightarrow,0,1)$ be a divisible operator residuated structure of finite height satisfying both the double negation law as well as the contraposition law. Then $\mathbb R\big(\mathbb P(\mathbf R)\big)=\mathbf R$ if and only if
\begin{align*}
\Min U(x,y\rightarrow0)\wedge y & \approx x\odot y, \\
 (x\rightarrow0)\vee\Max L(x,y) & \approx x\rightarrow y.
\end{align*}
\end{enumerate}
\end{theorem}

\begin{proof}
\
\begin{enumerate}[(i)]
\item If $\mathbb R(\mathbf P)=(P,\leq,\odot,\rightarrow,0,1)$ and $\mathbb P\big(\mathbb R(\mathbf P)\big)=(P,\leq,{}^*,0,1)$ then according to Theorem~\ref{th1} we have $x^*\approx x\rightarrow0\approx x'$.
\item If $\mathbb P(\mathbf R)=(R,\leq,{}',0,1)$ and $\mathbb R\big(\mathbb P(\mathbf R)\big)=(R,\leq,\bullet,\Rightarrow,0,1)$ then
\begin{align*}
            x' & \approx x\rightarrow0, \\
    x\bullet y & \approx\Min U(x,y\rightarrow0)\wedge y, \\
x\Rightarrow y & \approx(x\rightarrow0)\vee\Max L(x,y).
\end{align*}
\end{enumerate}
\end{proof}

According to Theorem~\ref{th2}, $\mathbb R(\mathbf P)$ contains the whole information on $\mathbf P$.

In the following we show that our construction of an operator residuated structure just introduced for othomodular posets of finite height can be easily modified for more general structures.

The following concepts were introduced for lattices in \cite{CL18a} and for posets in \cite{CL18b}.

\begin{definition}
\
\begin{itemize}
\item A {\em weakly orthomodular poset} is a bounded poset $(P,\leq,{}',0,1)$ with complementation such that $x,\in P$ and $x\leq y$ together imply that $x\vee(y\wedge x')$ is defined and that $x\vee(y\wedge x')=y$.
\item A {\em dually weakly orthomodular poset} is a bounded poset $(P,\leq,{}',0,1)$ with complementation such that $x,y\in P$ and $x\leq y$ together imply that $y\wedge(x\vee y')$ is defined and that $y\wedge(x\vee y')=x$.
\end{itemize}
\end{definition}

Hence, weakly orthomodular posets satisfy condition (A) and dually weakly orthomodular posets condition (B) mentioned after Definition~\ref{def3}.

It should be remarked that the complementation of a weakly orthomodular poset or a dually weakly orthomodular poset may neither be antitone nor an involution. Of course, a poset is orthomodular if and only if it is both weakly orthomodular and dually weakly orthomodular and if its complementation is an antitone involution.

In the following we show an example of a weakly orthomodular, respectively dually weakly orthomodular posets that is not orthomodular.

\begin{example}
The lattice $\mathbf W$ depicted in Figure~3
\vspace*{-2mm}
\begin{center}
\setlength{\unitlength}{7mm}
\begin{picture}(8,8)
\put(5,1){\circle*{.3}}
\put(1,3){\circle*{.3}}
\put(3,3){\circle*{.3}}
\put(5,3){\circle*{.3}}
\put(7,3){\circle*{.3}}
\put(3,5){\circle*{.3}}
\put(5,5){\circle*{.3}}
\put(7,5){\circle*{.3}}
\put(5,7){\circle*{.3}}
\put(5,1){\line(-2,1)4}
\put(5,1){\line(-1,1)2}
\put(5,1){\line(0,1)2}
\put(5,1){\line(1,1)2}
\put(3,5){\line(0,-1)2}
\put(3,5){\line(1,-1)2}
\put(5,5){\line(-1,-1)2}
\put(5,5){\line(1,-1)2}
\put(7,5){\line(-1,-1)2}
\put(7,5){\line(0,-1)2}
\put(5,7){\line(-1,-1)4}
\put(5,7){\line(0,-1)2}
\put(5,7){\line(1,-1)2}
\put(4.85,.3){$0$}
\put(.3,2.85){$a$}
\put(2.3,2.85){$b$}
\put(4.3,2.85){$c$}
\put(7.4,2.85){$d$}
\put(2.3,4.85){$e$}
\put(4.3,4.85){$f$}
\put(7.4,4.85){$g$}
\put(4.85,7.35){$1$}
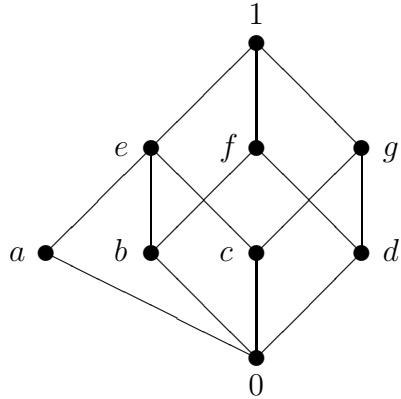
\put(-5.5,-.75){{\rm Fig.~3}. Weakly orthomodular poset being not dually weakly orthomodular}
\end{picture}
\end{center}
\vspace*{3mm}
with complementation defined by
\[
\begin{array}{c|ccccccccc}
x  & 0 & a & b & c & d & e & f & g & 1 \\
\hline
x' & 1 & g & g & f & e & d & c & b & 0
\end{array}
\]
is weakly orthomodular, but neither orthomodular nor dually weakly orthomodular. Hence the horizontal sum of $\mathbf W$ and the othomodular poset $\mathbf P$ from Figure~1 is a weakly orthomodular poset that is neither a lattice nor dually weakly orthomodular. On the contrary, the lattice $\mathbf D$ visualized in Figure~4
\vspace*{-2mm}
\begin{center}
\setlength{\unitlength}{7mm}
\begin{picture}(8,8)
\put(3,1){\circle*{.3}}
\put(1,3){\circle*{.3}}
\put(3,3){\circle*{.3}}
\put(5,3){\circle*{.3}}
\put(1,5){\circle*{.3}}
\put(3,5){\circle*{.3}}
\put(5,5){\circle*{.3}}
\put(7,5){\circle*{.3}}
\put(3,7){\circle*{.3}}
\put(3,1){\line(-1,1)2}
\put(3,1){\line(0,1)2}
\put(3,1){\line(1,1)4}
\put(1,5){\line(0,-1)2}
\put(1,5){\line(1,-1)2}
\put(3,5){\line(-1,-1)2}
\put(3,5){\line(1,-1)2}
\put(5,5){\line(-1,-1)2}
\put(5,5){\line(0,-1)2}
\put(3,7){\line(-1,-1)2}
\put(3,7){\line(0,-1)2}
\put(3,7){\line(1,-1)2}
\put(3,7){\line(2,-1)4}
\put(2.85,.3){$0$}
\put(.3,2.85){$a$}
\put(2.3,2.85){$b$}
\put(5.4,2.85){$c$}
\put(.3,4.85){$d$}
\put(2.3,4.85){$e$}
\put(5.4,4.85){$f$}
\put(7.4,4.85){$g$}
\put(2.85,7.35){$1$}
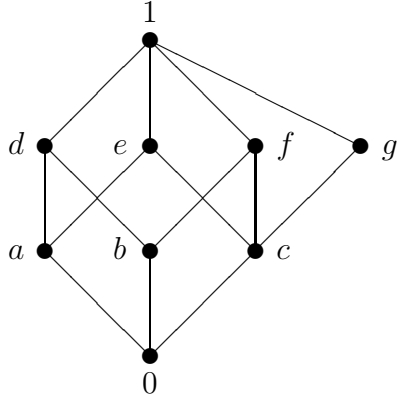
\put(-5.5,-.75){{\rm Fig.~4}. Dually weakly orthomodular poset being not weakly orthomodular}
\end{picture}
\end{center}
\vspace*{3mm}
with complementation defined by
\[
\begin{array}{c|ccccccccc}
x  & 0 & a & b & c & d & e & f & g & 1 \\
\hline
x' & 1 & f & e & d & c & b & a & a & 0
\end{array}
\]
is dually weakly orthomodular, but neither orthomodular nor weakly orthomodular. Hence the horizontal sum of $\mathbf D$ and $\mathbf P$ is a dually weakly orthomodular poset that is neither a lattice nor weakly orhomodular.
\end{example}

For such posets we can prove results concerning operator residuation without the contraposition law. However, we must assume the existence of operators $\odot$ and $\rightarrow$ since they are not automatically defined if only a weaker variant of the orthomodular law is assumed.

\begin{theorem}\label{th4}
Let $\mathbf P=(P,\leq,{}',0,1)$ be a bounded poset of finite height with complementation satisfying the identity $x''\approx x$ and define mappings $\odot,\rightarrow\colon P^2\rightarrow2^P\setminus\{\emptyset\}$ by
\begin{align*}
      x\odot y & :=\Min U(x,y')\wedge y, \\
x\rightarrow y & :=x'\vee\Max L(x,y)
\end{align*}
whenever defined {\rm(}$x,y\in P${\rm)}. Then $0'\approx1$ and $1'\approx0$ and the following hold:
\begin{enumerate}[{\rm(a)}]
\item Assume $\mathbf P$ to be weakly orthomodular. Then $\odot$ is well-defined. Assume $\rightarrow$ to be well-defined, too. Then $\mathbb R(\mathbf P):=(P,\leq,\odot,\rightarrow,0,1)$ satisfies all the conditions of an idempotent operator residuated structure satisfying the double negation law, $x\vee(x\rightarrow0)\approx1$ and $x\rightarrow0\approx x'$, only {\rm(ii)} has to be replaced by
\begin{enumerate}
\item[{\rm(ii')}] $x\odot y\sqsubseteq z$ implies $x\sqsubseteq y\rightarrow z$.
\end{enumerate}
\item Assume $\mathbf P$ to be dually weakly orthomodular. Then $\rightarrow$ is well-defined. Assume $\odot$ to be well-defined, too. Then $\mathbb R(\mathbf P):=(P,\leq,\odot,\rightarrow,0,1)$ satisfies all the conditions of an idempotent divisible operator residuated structure satisfying the double negation law and $x\rightarrow0\approx x'$, only {\rm(ii)} has to be replaced by
\begin{enumerate}
\item[{\rm(ii'')}] $x\sqsubseteq y\rightarrow z$ implies $x\odot y\sqsubseteq z$.
\end{enumerate}
\end{enumerate}
\end{theorem}

\begin{proof}
The proof is similar to that of Theorem~\ref{th1}. But we cannot use the fact that $'$ is antitone. Let $a,b,c\in P$. We have $0'\approx0\vee0'\approx1$ and $1'\approx1\wedge1'\approx0$.
\begin{enumerate}[(a)]
\item Because of Lemma~\ref{lem1}, $a\odot b$ is defined. By assumption also $a\rightarrow b$ is defined. We have $x\rightarrow0\approx x'\vee\Max L(x,0)=x'$ and since $x\leq1$ we have
\[
x\vee(x\rightarrow0)\approx x\vee x'\approx x\vee(1\wedge x')\approx1.
\]
The rest of the proof is completely analogous to that of Theorem~\ref{th1}.
\item Because of Lemma~\ref{lem1}, $a\rightarrow b$ is defined. By assumption also $a\odot b$ is defined. We have $x\rightarrow0\approx x'\vee\Max L(x,0)=x'$. The rest of the proof is completely analogous to that of Theorem~\ref{th1}.
\end{enumerate}
\end{proof}

Combining both assumptions (a) and (b) of Theorem~\ref{th4} we can prove a result similar to that of Theorem~\ref{th1}.

\begin{corollary}\label{cor2}
Let $\mathbf P=(P,\leq,{}',0,1)$ be a poset of finite height being both weakly orthomodular and dually weakly orthomodular and satisfying the identity $x''\approx x$ and define mappings $\odot,\rightarrow\colon P^2\rightarrow2^P\setminus\{\emptyset\}$ by
\begin{align*}
      x\odot y & :=\Min U(x,y')\wedge y, \\
x\rightarrow y & :=x'\vee\Max L(x,y)
\end{align*}
for all $x,y\in P$. Then $\odot$ and $\rightarrow$ are well-defined and $\mathbb R(\mathbf P):=(P,\leq,\odot,\rightarrow,0,1)$ is an idempotent and divisible operator residuated structure satisfying the double negation law and $x\rightarrow0\approx x'$.
\end{corollary}

\begin{proof}
This follows from Theorem~\ref{th4}.
\end{proof}

The following lattice is an example of a poset satisfying the conditions of Corollary~\ref{cor2}, but being not orthomodular.

\begin{example}
The lattice $\mathbf L$ depicted in Figure~5
\vspace*{-2mm}
\begin{center}
\setlength{\unitlength}{7mm}
\begin{picture}(12,8)
\put(3,1){\circle*{.3}}
\put(1,3){\circle*{.3}}
\put(3,3){\circle*{.3}}
\put(5,3){\circle*{.3}}
\put(9,3){\circle*{.3}}
\put(3,5){\circle*{.3}}
\put(7,5){\circle*{.3}}
\put(9,5){\circle*{.3}}
\put(11,5){\circle*{.3}}
\put(9,7){\circle*{.3}}
\put(3,1){\line(-1,1)2}
\put(3,1){\line(0,1)4}
\put(3,1){\line(1,1)2}
\put(3,1){\line(3,1)6}
\put(1,3){\line(1,1)2}
\put(1,3){\line(3,1)6}
\put(3,3){\line(3,1)6}
\put(5,3){\line(-1,1)2}
\put(5,3){\line(3,1)6}
\put(9,3){\line(-1,1)2}
\put(9,3){\line(0,1)4}
\put(9,3){\line(1,1)2}
\put(3,5){\line(3,1)6}
\put(7,5){\line(1,1)2}
\put(11,5){\line(-1,1)2}
\put(2.85,.3){$0$}
\put(.4,2.85){$a$}
\put(2.4,2.85){$b$}
\put(4.3,2.85){$c$}
\put(8.85,2.3){$d$}
\put(2.85,5.35){$e$}
\put(7.35,4.85){$f$}
\put(9.25,4.85){$g$}
\put(11.3,4.85){$h$}
\put(8.85,7.35){$1$}
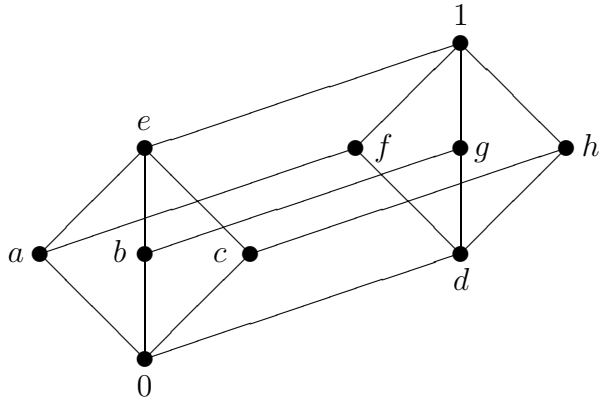
\put(-5.4,-.75){{\rm Fig.~5}. Weakly orthomodular and dually weakly orthomodular lattice whose complementation}
\put(-3.4,-1.45){is a non-antitone involution}
\end{picture}
\end{center}
\vspace*{10mm}
with complementation defined by
\[
\begin{array}{c|cccccccccc}
x  & 0 & a & b & c & d & e & f & g & h & 1 \\
\hline
x' & 1 & g & h & f & e & d & c & a & b & 0
\end{array}
\]
is modular and hence both weakly orthomodular and dually weakly orthomodular and its complementation is an involution but not antitone since $a\leq f$, but $f'=c\not\leq g=a'$. Hence the horizontal sum of $\mathbf L$ with the orthomodular poset $\mathbf P$ from Figure~1 is a weakly orthomodular and dually weakly orthomodular poset that is neither a lattice nor an orthomodular poset.
\end{example}

Similarly, as it was done above in the case of orthomodular posets, we can prove the converse of Corollary~\ref{cor2}.

\begin{theorem}
Let $\mathbf R=(R,\leq,\odot,\rightarrow,0,1)$ be a divisible operator residuated structure of finite height satisfying the double negation law and define
\[
x':=x\rightarrow0
\]
for all $x\in R$. Then $\mathbb P(\mathbf R):=(R,\leq,{}',0,1)$ is a dually weakly orthomodular poset of finite height satisfying the identity $x''\approx x$.
\end{theorem}

\begin{proof}
Let $a,b\in R$. Since $\mathbf R$ satisfies the double negation law, $'$ is an involution on $(R,\leq)$. Now assume $a\leq b$. Because of (vi) of Definition~\ref{def1} we have $b\rightarrow a=(b\rightarrow0)\vee a=b'\vee a\geq b'$. According to (v) of Definition~\ref{def1} we conclude $(b\rightarrow a)\odot b=(b\rightarrow a)\wedge b$ which together with divisibility yields
\[
b\wedge(a\vee b')=(b'\vee a)\wedge b=(b\rightarrow a)\wedge b=(b\rightarrow a)\odot b=\Max L(b,a)=a,
\]
i.e.\ $\mathbb P(\mathbf R)$ is dually weakly orthomodular. Because of $0\leq a$ we obtain $0=a\wedge(0\vee a')=a\wedge a'$. According to (vii) of Definition~\ref{def1} we have $a\vee a'=a\vee(a\rightarrow0)=1$ completing the proof of the theorem.
\end{proof}

The following result shows that the correspondence between posets $(P,\leq,{}',0,1)$ of finite height being both weakly orthomodular and dually weakly orthomodular and satisfying the identity $x''\approx x$ on the one side and operator residuated structures of finite height on the other side is almost one-to-one.

\begin{corollary}\label{cor1}
Let $\mathbf P=(P,\leq,{}',0,1)$ be a poset of finite height being both weakly orthomodular and dually weakly orthomodular and satisfying the identity $x''\approx x$. Then $\mathbb P\big(\mathbb R(\mathbf P)\big)=\mathbf P$.
\end{corollary}

\begin{proof}
If $\mathbb R(\mathbf P)=(P,\leq,\odot,\rightarrow,0,1)$ and $\mathbb P\big(\mathbb R(\mathbf P)\big)=(P,\leq,{}^*,0,1)$ then according to Corollary~\ref{cor2} we have $x^*\approx x\rightarrow0\approx x'$. The rest of the proof is evident.
\end{proof}

According to Corollary~\ref{cor1}, $\mathbb R(\mathbf P)$ contains the whole information on $\mathbf P$.

Authors' addresses:

Ivan Chajda \\
Palack\'y University Olomouc \\
Faculty of Science \\
Department of Algebra and Geometry \\
17.\ listopadu 12 \\
771 46 Olomouc \\
Czech Republic \\
ivan.chajda@upol.cz

Helmut L\"anger \\
TU Wien \\
Faculty of Mathematics and Geoinformation \\
Institute of Discrete Mathematics and Geometry \\
Wiedner Hauptstra\ss e 8-10 \\
1040 Vienna \\
Austria, and \\
Palack\'y University Olomouc \\
Faculty of Science \\
Department of Algebra and Geometry \\
17.\ listopadu 12 \\
771 46 Olomouc \\
Czech Republic \\
helmut.laenger@tuwien.ac.at
\end{document}